\documentclass{amsart}
\usepackage{hyperref}
\usepackage{amssymb}
\usepackage{fancyhdr}		
\usepackage{color}

\usepackage[T1]{fontenc}
\usepackage[latin1]{inputenc}
\usepackage{url}	
\usepackage{graphicx}	
\usepackage{amssymb}
\usepackage{amsmath}
\usepackage{amsthm}
\usepackage[T1]{fontenc}
\usepackage{ae}

\newtheorem{theo}{Theorem}[section]
\newtheorem{cor}[theo]{Corollary }
\newtheorem{lemm}[theo]{Lemma}
\newtheorem{fact}[theo]{Fact}
\newtheorem{defi}{Definition}[section]

\newtheorem{prop}[theo]{Proposition}
\newtheorem{Remark}[theo]{Remark}

\def\R{\mathbb{R}}

\def\SO{{\sf{SO}}}
 
\def\SL{{\sf{SL}}}

\title[]{Pseudo-Conformal actions  of the Möbius group}

\date{\today}
\author{M. Belraouti}
\address{Mehdi Belraouti \newline
Faculté de Mathématiques,\\
USTHB, BP 32, El-Alia,\\
16111 Bab-Ezzouar, Alger (Algeria)}
\email{mbelraouti@usthb.dz}

\author{M. Deffaf}
\address{Mohamed Deffaf \newline
Faculté de Mathématiques,\\
USTHB, BP 32, El-Alia,\\
16111 Bab-Ezzouar, Alger (Algeria)}
\email{mdeffaf@usthb.dz}

\author{Y. Raffed} 

\address{Yazid raffed \newline
Faculté de Mathématiques,\\
USTHB, BP 32, El-Alia,\\
16111 Bab-Ezzouar, Alger (Algeria) }
\email{yazidsaid.raffed@usthb.edu.dz}

\author{A. Zeghib}
\address{Abdelghani Zeghib \newline
UMPA, ENS de Lyon, France }
\email{abdelghani.zeghib@ens-lyon.fr}
\begin{document}
\maketitle

\noindent{\bf Abstract.}
We study compact connected pseudo-Riemannian manifolds $(M,g)$ on which the conformal group $\operatorname{Conf}(M,g)$ acts essentially and transitively. We prove, in particular,  that if the non-compact semi-simple part of $\operatorname{Conf}(M,g)$ is the Möbius group, then $(M,g)$ is conformally flat.

\tableofcontents


\section{Introduction}
A pseudo-Riemannian manifold is a differentiable manifold $M$ endowed with a pseudo-Riemannian metric $g$ of signature $(p,q)$. Two 
 metrics $g_{1}$ and $g_{2}$ on $M$ are said to be conformally equivalent if and only if $g_{1}=\exp(f)g_{2}$ where $f$ is $C^{\infty}$ function. A conformal structure is then 
an equivalence class $[g]$ of a pseudo-Riemannian metric $g$  and a conformal manifold is a manifold endowed with a pseudo-Riemannian conformal structure. A remarkable family of  conformal manifolds is given by the conformally flat ones. These are pseudo-Riemannian conformal manifolds that are locally conformally diffeomorphic (i.e preserving the conformal structures) to  the Minkowski space $\mathbb{R}^{p,q}$ i.e the vector space $\mathbb{R}^{p+q}$ endowed with the pseudo-Riemannian metric $-dx_{0}^{2}-...-dx_{p-1}^{2}+dy_{0}^{2}+...+dy_{q-1}^{2}$.

The conformal group $\operatorname{Conf}(M,g)$ is the group of transformations that preserve the conformal structure $[g]$. It is said to be essential if there is no metric in the conformal class of $g$  for which it acts isometrically. In the Riemannian case, the sphere $\mathbb{S}^{n}$ is a compact conformally flat manifold with an essential conformal group. 

The Einstein universe $\operatorname{Ein}^{p,q}$ is the equivalent model of the standard conformal sphere in the pseudo-Riemannian setting. It admits a two-fold covering conformally equivalent to the product $\mathbb{S}^{p}\times \mathbb{S}^{q}$ endowed with the conformal class of $-g_{\mathbb{S}^{p}}\oplus g_{\mathbb{S}^{q}}$. It is conformally flat and its conformal group, which is in fact the pseudo-Riemannian Möbius group $\operatorname{O}(p+1,q+1)$, is essential. Actually the Einstein universe is the flat model of conformal pseudo-Riemannian geometry. This is essentially due to the fact that the Minkowski space embeds conformally as a dense open subset of the Einstein universe $\operatorname{Ein}^{p,q}$ and in addition to the Liouville theorem  asserting that conformal local diffeomorphisms on $\operatorname{Ein}^{p,q}$ are unique restrictions of elements of $\operatorname{O}(p+1,q+1)$. Hence a manifold is conformally flat if and only if it admits a $\left(\operatorname{O}(p+1,q+1),\operatorname{Ein}^{p,q}\right)$-structure.

In the sixties A. Lichn\'erowicz conjectured that among compact Riemannian manifolds, the sphere is the only essential conformal structure. This was generalised and  proved independently by Obatta and Ferrand  (see \cite{Obata}, \cite{Ferrand}). In the pseudo-Riemannian case, a similar question, called the pseudo-Riemannian Lichn\'erowicz conjecture, was raised by  D'Ambra and Gromov \cite{Gromov}. Namely, if a compact pseudo-Riemannian conformal manifold is essential then it is conformally flat. This was disproved by Frances  see \cite{Francesun}, \cite{frances2}.  

The present article is the first of a series on  the pseudo-Riemannian Lichn\'erowicz conjecture in a homogeneous setting \cite{BDRZdeux, BDRZ}. The general non homogeneous case, 
but with signature 
restrictions,  was amply   studied by  Zimmer, Bader, Nevo,  Frances, Zeghib,  Melnick and Pecastaing (see \cite{Zimmerun}, \cite{Bader}, \cite{Francesquatre}, \cite{Pecastaingun}, \cite{Pecastaingdeux}, \cite{Vincenttrois}, \cite{Vincentquatre}). Let us also quote \cite{leistner2023conformal} as a recent work in the Lorentz case.

We are investigating  in this first part  the case where the non-compact semi-simple part of the conformal group is locally isomorphic to the Möbius group $\operatorname{SO}(1,n+1)$.  
More  exactly, we prove the following  classification result. This Möbius situation   will actually play a central role   towards the general case treated in  \cite{BDRZdeux}.

\begin{theo}
\label{theoremdeux}
Let $(M,[g])$ be a conformal connected compact pseudo-Riemannian manifold. We suppose that there exists $G$ a subgroup of  the conformal group $\operatorname{Conf}(M,g)$ acting  essentially and transitively on $(M,[g])$. We suppose moreover that the non-compact semi-simple part of  $G$  is locally isomorphic to the Möbius group $\operatorname{SO}(1,n+1)$. Then $(M,[g])$ is conformally flat.  More precisely $(M,[g])$ is conformally equivalent to 
\begin{itemize}
\item The conformal Riemannian $n-$sphere or; 
\item Up to a cover, the Einstein universe $\operatorname{Ein}^{1,1}$ or;
\item Up to a finite cover, the Einstein universe $\operatorname{Ein}^{3,3}$.
\end{itemize}
\end{theo}

\begin{Remark} It turns out that, in the first and third cases,  the acting group $G$ is locally isomorphic to the Môbius group, that is,  $G$ 
is simple. In the second case, the universal cover $\tilde{G}$ is a subgroup  of 
$\widetilde{\SL(2, \R)} \times \widetilde{\SL(2, \R)}$. It can in particular be 
$\widetilde{\SL(2, \R)} \times \widetilde{\SO(2)}$.

\end{Remark}

\section{Preliminaries}

\subsection{Notations}
Throughout this paper $(M,g)$  will be a compact connected pseudo-Riemannian manifold of dimension $n$ endowed with a transitive and essential action of the conformal group $G=\operatorname{Conf}(M,g)$. We  suppose without loss of generality that $G$ is connected.

Fix a point $x$ in $M$ and denote by $H=\operatorname{Stab}(x)$ its stabilizer in $G$. Denote respectively by $\mathfrak{g}$, $\mathfrak{h}$ the Lie algebras of $G$ and $H$. Let $\mathfrak{g}=\mathfrak{s}\ltimes \mathfrak{r}$ be a Levi decomposition of $\mathfrak{g}$, where $\mathfrak{s}$ is semi-simple and $\mathfrak{r}$ is the solvable radical of $\mathfrak{g}$. Denote by $\mathfrak{s}_{nc}$  the non-compact semi-simple factor of $\mathfrak{s}$, by $\mathfrak{s}_{c}$  the compact one and 
let $\mathfrak{n}$ be the nilpotent radical of $\mathfrak{g}$. Note that $\mathfrak{n}$ is an ideal of $\mathfrak{g}$. Let us denote respectively by  $S$, $S_{nc}$, $S_{c}$, $R$ and $N$ the connected Lie sub-groups of $G$ associated to $\mathfrak{s}$, $\mathfrak{s}_{nc}$, $\mathfrak{s}_{c}$, $\mathfrak{r}$ and $\mathfrak{n}$.

Let $\mathfrak{a}$ be a Cartan subalgebra of $\mathfrak{s}$ associated  with a Cartan involution $\Theta$. Consider $\mathfrak{s}=\mathfrak{s}_{0} \oplus  \bigoplus_{\alpha \in \Delta} \mathfrak{s}_{\alpha}=   \mathfrak{a}\oplus \mathfrak{m}\oplus  \bigoplus_{\alpha \in \Delta} \mathfrak{s}_{\alpha}$ the  root space decomposition  of $\mathfrak{s}$, where $\Delta$ is the set of roots of $(\mathfrak{s},\mathfrak{a})$. Denote respectively by  $\Delta^{+}$, $\Delta^{-}$ the set of  positive and negative roots of $\mathfrak{s}$ for some chosen notion of positivity on $\mathfrak{a}^{*}$. Then $\mathfrak{s}=  \mathfrak{s}_{-}\oplus\mathfrak{a}\oplus \mathfrak{m} \oplus \mathfrak{s}_{+}$, where $\mathfrak{s}_{+}=\bigoplus_{\alpha\in \Delta^{+}} \mathfrak{s}_{\alpha}$ and $\mathfrak{s}_{-}=\bigoplus_{\alpha\in \Delta^{-}} \mathfrak{s}_{\alpha}$.

For every  $\alpha \in \mathfrak{a}^{*}$, consider
$$\mathfrak{g}_{\alpha} = \{ X\in \mathfrak{g} , \forall H\in \mathfrak{a} : ad_{H}(X)=\alpha(H)X \}.$$
We say that $\alpha$ is a weight if $\mathfrak{g}_{\alpha}\neq 0$. In this case $\mathfrak{g}_{\alpha}$ is its associated weight space. 
As $\left[\mathfrak{g},\mathfrak{r}\right]\subset \mathfrak{n}$ (see \cite[Theorem~13]{J}) then, for every $\alpha\neq 0$,  $\mathfrak{g}_{\alpha}=\mathfrak{s}_{\alpha}\oplus\mathfrak{n}_{\alpha}$, where $$\mathfrak{n}_{\alpha} = \{ X\in \mathfrak{n} , \forall H\in \mathfrak{a} : ad_{H}(X)=\alpha(H)X \}.$$ 
Moreover, the commutativity of $\mathfrak{a}$ together with the fact that finite dimensional representations of a semi-simple Lie algebra preserve the Jordan decomposition implies that elements of $\mathfrak{a}$ are simultaneously diagonalisable in some basis of $\mathfrak{g}$. Thus $\mathfrak{g}=\mathfrak{g}_{0} \oplus  \bigoplus_{\alpha \neq 0} \mathfrak{g}_{\alpha}$. 

Finally we will denote respectively by $A$, $S_{+}$  the connected Lie subgroups of $G$ corresponding to $\mathfrak{a}$ and $\mathfrak{s}_{+}$.

\subsection{General facts}
We will prove some general results about the conformal group $G$. We start with the following general fact:
\begin{prop}
\label{propo00}
We have that $\left[\mathfrak{s}, \mathfrak{n}\right]=\left[\mathfrak{s}, \mathfrak{r}\right]$. In particular the sub-algebra $\mathfrak{s}\ltimes \mathfrak{n}$ is an ideal of $\mathfrak{g}$. 
\end{prop}
\begin{proof}
For this, let us consider the semi-simple  $S-$representation in $GL(\mathfrak{r})$. It preserves $\mathfrak{n}$ and thus has a supplementary invariant subspace. But $\left[\mathfrak{g},\mathfrak{r}\right]\subset \mathfrak{n}$ so automorphisms of $\mathfrak{r}$ act trivially on $\mathfrak{r}/\mathfrak{n}$ and hence $[\mathfrak{s},\mathfrak{g}]\subset\mathfrak{s}\oplus[\mathfrak{s},\mathfrak{n}] \subset\mathfrak{s}\ltimes \mathfrak{n}$. We deduce that $\mathfrak{s}\ltimes \mathfrak{n}$ is an ideal of $\mathfrak{g}$.
\end{proof}

Next we will prove:
\begin{prop}
\label{Porpositionune}
The non-compact semi-simple factor $S_{nc}$ of $S$ is non trivial.
\end{prop}
Let us first start with the following simple observation:
\begin{prop}
If a conformal diffeomorphism $f$ of $(M,g)$ preserves a volume form $\omega$ on $M$, then it preserves a metric in the conformal class of $g$.
\end{prop}
\begin{proof}
Let $f$ be a diffeomorphism preserving the conformal class $\left[g\right]$ and a volume form $\omega$ on $M$. Denote by $\omega_{g}$  the volume form defined on $M$ by the metric $g$. On the one hand, there exists a $C^{\infty}$ real function $\phi$ such that $\omega=e^{\phi}\omega_{g}$. Hence $\omega$ is the volume form defined by  the metric $e^{\frac{2\phi}{n}}g$. On the other hand, we have $f^{*}e^{\frac{2\phi}{n}}g=e^{\psi }e^{\frac{2\phi}{n}}g$, for some $C^{\infty}$ function $\psi$. Thus  $f^{*}\omega=e^{\frac{n}{2}\psi}\omega$. But, $f$ preserves the volume form $\omega$, so $\psi=0$ which means that $f$ preserves the metric $e^{\frac{2\phi}{n}}g$. 
\end{proof}
As a consequence we get:
\begin{cor}
\label{cor80}
The conformal group $G$   preserves no volume form on $M$.
\end{cor}

Assume that the non-compact semi-simple factor $S_{nc}$ is trivial. Then by \cite[Corollary~4.1.7]{Zimme} the  group $G$ is amenable. So it preserves a regular Borel measure $\mu$ on the compact manifold $M$. It is in particular a quasi-invariant measure with associated rho-function $\rho_{1}=1$ (in the sense of \cite{BHV}). Let now $\omega_{g}$ be the volume form corresponding to the metric $g$. As the group $G$ acts conformally and the action is $C^{\infty}$, the measure $\omega_{g}$ is also  quasi-invariant with $C^{\infty}$ rho-function $\rho_{2}$ (see \cite[Theorem~B.1.4]{BHV}). 
Again by \cite[Theorem~B.1.4]{BHV}, the measures $\mu$ and $\omega_{g}$ are equivalent and $\frac{d\mu}{d\omega_{g}}= \frac{1}{\rho_{2}}$. This shows that $\mu$ is a volume form. Then one use Corollary \ref{cor80} to get the  Proposition \ref{Porpositionune}.

In the general case the essentiality of the action ensure the non discreetness of  the stabilizer $H$.
\begin{prop}
The stabilizer $H$ is not discrete.
\end{prop}
\begin{proof}
If it was not the case then $H$ would be a uniform lattice in $G$. But as the action is essential, there is an element $h\in H$ that does not preserve the metric on $\mathfrak{g}/\mathfrak{h}$. So $\left|\operatorname{det}\left(\operatorname{Ad}_{h}\right)\right|\neq 1$ contradicting the unimodilarity of $G$.
\end{proof}

To finish this part let us prove the two following important Lemmas that will be used later in the paper:

\begin{lemm}
\label{Lemmaimport}
Let $\operatorname{\pi}:S_{nc}\longrightarrow \operatorname{GL}(V)$ be a linear representation of $S_{nc}$ into a linear space $V$. Then, the compact orbits of $S_{nc}$ are trivial.
\end{lemm}
\begin{proof}
Assume that $S_{nc}$ has a compact orbit $\mathcal{C}\subset V$. Then the convex envelope $\operatorname{Conv}(\mathcal{C}\cup -\mathcal{C})$ is an $S_{nc}-$invariant compact convex symmetric set with non empty interior. Thus the action of $S_{nc}$ preserves the Minkowski gauge $\left\|.\right\|$ (which is in fact a norm) of $\operatorname{Conv}(\mathcal{C}\cup -\mathcal{C})$. But $\operatorname{Isom}\left(\operatorname{Conv}(\mathcal{C}\cup -\mathcal{C}), \left\|.\right\|\right)$ is compact. So the restriction of the representation $\pi$ to $\operatorname{Conv}(\mathcal{C}\cup -\mathcal{C})$ gives rise to an homomorphism from a semi-simple group with no compact factor  to a compact group and hence is trivial.
\end{proof}
\begin{lemm}
\label{Lemmaimportprime}
A linear representation $\operatorname{\pi}:\mathfrak{s}_{nc}\longrightarrow \operatorname{gl}(V)$ of $\mathfrak{s}_{nc}$ into a linear space $V$ is completely determined by its restriction to $\mathfrak{a}\oplus\mathfrak{m}\oplus  \mathfrak{s}_{+}$. More precisely, $\operatorname{\pi}_{\mathfrak{s}_{nc}}(V)=\operatorname{Vect}\left(\operatorname{\pi}_{\mathfrak{a}\oplus\mathfrak{m}\oplus  \mathfrak{s}_{+}}\left(V\right)\right)$. 
\end{lemm}
\begin{proof}
It is in fact sufficient to show that $\operatorname{\pi}_{\mathfrak{s}_{-}}\left(V\right)\subset \operatorname{Vect}\left(\operatorname{\pi}_{\mathfrak{a}\oplus\mathfrak{s}_{+}}\left(V\right)\right)$. For that, fix $x\in \mathfrak{s}_{-\alpha}\subset \mathfrak{s}_{-}$ and let $a\in \mathfrak{a}$ such that $\mathbb{R}x\oplus \mathbb{R}a\oplus \mathbb{R}\Theta(x)\cong \mathfrak{sl}(2,\mathbb{R})$ (see for example \cite[Proposition~6.52]{K}). Thus the restriction of $\operatorname{\pi}$ to $\mathbb{R}x\oplus \mathbb{R}a\oplus \mathbb{R}\Theta(x)$ is isomorphic to a linear representation of $\mathfrak{sl}(2,\mathbb{R})$ into $V$. Using  Weyl Theorem we can assume without loss of generality that this last is irreducible. But irreducible linear representations of $\mathfrak{sl}(2,\mathbb{R})$ into $V$ are unique up to isomorphism (see for instance \cite[Theorem~4.32]{Hall}). It is then easy to check that they verify $\operatorname{\pi}(x)(V)\subset \operatorname{Vect}\left(\operatorname{\pi}_{\mathbb{R}a\oplus \mathbb{R}\Theta(x)}\left(V\right)\right)$ (see \cite[Examples~4.2]{Hall}). This finishes the proof.
\end{proof}

\section{Lie algebra formulation}

\subsection{Enlargement of the isotropy group}
\label{subsection}
As the manifold $G/H$ is compact, the isotropy subgroup $H$ is a uniform subgroup of $G$. If $H$ was discrete then it is a uniform lattice and in this case $G$ would be  unimodular. In the non discrete case, this imposes strong restrictions on the group $H$. When $H$ and $G$ are both complex algebraic it is equivalent to being parabolic i.e contains maximal solvable connected subgroup of $H$. In the real case, Borel and Tits \cite{BT} proved that an algebraic group $H$ of a real linear algebraic group $G$ is uniform if it contains a maximal connected triangular subgroup of $G$. Recall that a subgroup of $G$ (respectively a sub-algebra of $\mathfrak{g}$) is said to be triangular if, in some real basis of $\mathfrak{g}$,  its image under the adjoint representation is triangular. 

Let $H^{*}=\operatorname{Ad}^{-1}\left(\overline{\operatorname{Ad}(H)}^{Zariski}\right)$ be the smallest algebraic Lie subgroup of $G$ containing $H$. By \cite[Corollary~5.1.1]{GO}, the Lie algebra $\mathfrak{h}^{*}$ of $H^{*}$ contains a maximal triangular sub-algebra of $\mathfrak{g}$. The sub-algebra $\left(\mathfrak{a} \oplus \mathfrak{s}_{+}\right)\ltimes \mathfrak{n}$ being triangular, we get the following fact:
\begin{fact}
\label{Fact9}
Up to conjugacy,  the sub-algebra $h^{*}$ contains $\left(\mathfrak{a} \oplus\mathfrak{s}_{+}\right)\ltimes \mathfrak{n}$. 
\end{fact}

Consider the vector space $\operatorname{Sym}(\mathfrak{g})$ of bilinear symmetric forms on $\mathfrak{g}$. The group $G$ acts naturally on $\operatorname{Sym}(\mathfrak{g})$ by $g.\Phi(X,Y)=\Phi(Ad_{g^{-1}}X,Ad_{g^{-1}}Y)$.
Let $\left\langle .,.\right\rangle$ be the bilinear symmetric form on $\mathfrak{g}$ defined by $$\left\langle X,Y\right\rangle=g\left(X^{*}(x),Y^{*}(x)\right),$$ where $g$ is the pseudo-Riemannian metric, $X^{*}$, $Y^{*}$ are the fundamental vector fields associated to $X$ and $Y$ and $x$ is the point fixed previously. It is a degenerate symmetric form with kernel equal to $\mathfrak{h}$. 

Let $P$ be the subgroup of $G$ preserving the conformal class of $\left\langle .,.\right\rangle$. It is an algebraic group containing $H$  and  normalizing the sub-algebra $\mathfrak{h}$. In particular, it contains  $H^{*}$: the smallest algebraic group containing $H$. Using Fact \ref{Fact9} we get 
that up to conjugacy, the Lie algebra $\mathfrak{p}$ of $P$ contains $\left(\mathfrak{a} \oplus\mathfrak{s}_{+}\right)\ltimes \mathfrak{n}$.

\begin{prop}
\label{propositiondeux}
The Cartan sub-group $A$ does not preserve the metric $\left\langle .,.\right\rangle$.
\end{prop}
\begin{proof}
First as $\mathfrak{h}$ is an ideal of $\mathfrak{p}$ then by taking quotient of both $P$ and $H$ by $H^{°}$, we can suppose that $H$ is  a uniform lattice of $P$ and in particular that $P$ is unimodular.

Assume that $A$ preserves the metric $\left\langle .,.\right\rangle$. On the one hand, the groups $S_{+}$ and $N$ preserve the conformal class of $\left\langle .,.\right\rangle$. On the other hand, they  act on $\operatorname{Sym}(\mathfrak{g})$ by unipotent elements. So the groups $A$, $S_{+}$, and $N$  preserve the metric $\left\langle .,.\right\rangle$. But by Iwasawa  decomposition  $\left(A\ltimes S_{+}\right)$ is co-compact in $S$. Thus the $S-$orbit of $\left\langle .,.\right\rangle$ is compact in $\operatorname{Sym}(\mathfrak{g})$ and hence trivial by Lemma \ref{Lemmaimport}. Therefore  $S$ and $N$  are subgroups of $P$. This implies that for any $p\in P$,  $\left|\operatorname{det}\left(\operatorname{Ad}_{p}\right)_{\vert{\mathfrak{g}/\mathfrak{p}}}\right|=1$. Indeed, the action of $G$ on $(\mathfrak{s}_{c}+\mathfrak{r})/\mathfrak{n}$ factors trough the product of the action of $S_{c}$ on $\mathfrak{s}_{c}$ by the trivial action on $\mathfrak{r}/\mathfrak{n}$. As $P$ contains $S$ and  $N$, its action on $\mathfrak{g}/\mathfrak{p}$ is a quotient of the action of  $S_{c}$  on $\mathfrak{s}_{c}$. But $S_{c}$ is compact, thus it preserves some positive definite scalar product and hence the determinant $\left|\operatorname{det}\left(\operatorname{Ad}_{p}\right)_{\vert{\mathfrak{g}/\mathfrak{p}}}\right|=1$.


Now let $h\in H$ such that $\operatorname{Ad}_{h}$ does not preserve $\left\langle .,.\right\rangle$. We have that $$1\neq \left|\operatorname{det}\left(\operatorname{Ad}_{h}\right)_{\vert{\mathfrak{g}/\mathfrak{h}}}\right|=\left|\operatorname{det}\left(\operatorname{Ad}_{h}\right)_{\vert{\mathfrak{g}/\mathfrak{p}}}\right|\left| \operatorname{det}\left(\operatorname{Ad}_{h}\right)_{\vert{\mathfrak{p}/\mathfrak{h}}}\right|$$

Finally we get $\left|\operatorname{det}\left(\operatorname{Ad}_{h}\right)_{\vert{\mathfrak{p}/\mathfrak{h}}}\right|\neq 1$ which contradicts the unimodularity of $P$.
\end{proof}

\subsection{Distortion}
The group $P$ preserves the conformal class of $\left\langle .,.\right\rangle$. There exists thus an homomorphism $\delta: P\rightarrow \mathbb{R}$ such that: for every $p\in P$ and every $u,v\in \mathfrak{g}$, 
\begin{equation}
\label{equationdeux}
\left\langle\operatorname{Ad}_{p}(u), \operatorname{Ad}_{p}(v)\right\rangle=e^{\delta(p)}\left\langle u,v\right\rangle=\left|\operatorname{det}\left(\operatorname{Ad}_{p}\right)_{\vert{\mathfrak{g}/\mathfrak{h}}}\right|^{\frac{2}{n}}\left\langle u,v\right\rangle
\end{equation}

In particular if $p\in P$ preserves the metric then $\delta(p)=0$ and 

\begin{equation}
\label{equationquatre}
\left\langle\operatorname{Ad}_{p}(u), \operatorname{Ad}_{p}(v)\right\rangle=\left\langle u,v\right\rangle
\end{equation}

Or equivalently 

\begin{equation}
\label{equationcinq}
\left\langle\operatorname{ad}_{p}(u),v\right\rangle + \left\langle u, \operatorname{ad}_{p}(v)\right\rangle=0 
\end{equation}

It follows that if the action of $p\in P$ on $\mathfrak{g}$ is unipotent then $\delta(p)=0$.  Therefore, the homomorphism $\delta$ is trivial on $S_{-}$ and $N$ but not on $A$ by Proposition \ref{propositiondeux}. We continue to denote by $\delta$ the restriction of $\delta$ to $A$. We can see it alternatively as a linear form $\delta:\mathfrak{a}\rightarrow \mathbb{R}$, called \textbf{distortion}, verifying: for every $a\in \mathfrak{a}$ and every $u,v\in \mathfrak{g}$, 

\begin{equation}
\label{equationtrois}
\left\langle\operatorname{ad}_{a}(u),v\right\rangle + \left\langle u, \operatorname{ad}_{a}(v)\right\rangle=\delta(a)\left\langle u,v\right\rangle
\end{equation}

\begin{defi}
Two weights spaces $\mathfrak{g}_{\alpha}$ and $\mathfrak{g}_{\beta}$ are said to be paired if they are not $\left\langle .,.\right\rangle-$orthogonal. 
\end{defi}
\begin{defi}
A weight $\alpha$ is a non-degenerate weight if $\mathfrak{g}_{\alpha}$ is not contained in $\mathfrak{h}$.
\end{defi}

\begin{defi}
We say that a subalgebra $\mathfrak{g}'$ is a modification of $\mathfrak{g}$ if $\mathfrak{g}'$ projects surjectively on $\mathfrak{g}/\mathfrak{h}$. In this case $\mathfrak{g}'/\left(\mathfrak{g}'\cap \mathfrak{h}\right)=\mathfrak{g}/\mathfrak{h}$. 
\end{defi}

\begin{prop}
\label{proposition8}
If the weight space $\mathfrak{g}_{0}$ is degenerate then up to modification, $\mathfrak{g}$ is semi-simple and  $M=G/H$ is conformally flat.
\end{prop}
\begin{proof}
On the one hand, $\mathfrak{g}_{0}\subset\mathfrak{h}$ implies that $\mathfrak{a}\subset \mathfrak{h}$. As $\mathfrak{h}$ is an ideal of $\mathfrak{p}$, we get that $\mathfrak{s}_{+}=\left[\mathfrak{s}_{+},\mathfrak{a}\right]\subset \mathfrak{h}$. On the other hand,  $\mathfrak{r}\subset \mathfrak{g}_{0}+\mathfrak{n}\subset \mathfrak{g}_{0}+\left[\mathfrak{n},\mathfrak{a}\right]\subset \mathfrak{h}$. Thus, up to modification, we can assume  that $\mathfrak{g}$ is semi-simple and that $\mathfrak{h}$ contains $\mathfrak{a}+\mathfrak{s}_{+}$. 

Now let $\alpha_{max}$ be the highest positive root and let $X\in \mathfrak{g}_{\alpha_{max}}$. Then $d_{1}e^{X}:\mathfrak{g}/\mathfrak{h}\rightarrow \mathfrak{g}/\mathfrak{h}$ is trivial. Yet $e^{X}$ is not trivial. We conclude using \cite[Theorem~1.4]{CharlesMelnick}.
\end{proof}

A direct consequence of Equation \ref{equationtrois}, is that if $\mathfrak{g}_{\alpha}$ and $\mathfrak{g}_{\beta}$ are paired then $\alpha+\beta=\delta$. This shows that if $\alpha$ is a non-degenerate weight then $\delta-\alpha$ is also a non-degenerate weight. In particular if $0$ is a non-degenerate weight, then  $\mathfrak{g}_{0}$ and $\mathfrak{g}_{\delta}$ are paired and hence $\delta$ is a non-degenerate weight. In fact:

\begin{prop}
\label{proposition9}
If $0$ is a non-degenerate weight then $\delta$ is a root. Moreover $\mathfrak{s}_{\delta}\not\subset \mathfrak{h}$.
\end{prop}

\begin{proof}
\label{preuveune}
First we will prove that the subalgebras $\mathfrak{a}$ and $\mathfrak{n}_{\delta}$ are $\left\langle .,.\right\rangle-$orthogonal.  Let $a\in \mathfrak{a}$ such that $\delta(a)\neq 0$. Using Equation \ref{equationtrois} for $a$, $u=a$ and $v\in \mathfrak{n}_{\delta}$, we get, $\left\langle a, \operatorname{ad}_{a}(v)\right\rangle=\delta(a)\left\langle a,v\right\rangle$. But $v$ preserves $\left\langle .,.\right\rangle$, thus by Equation \ref{equationcinq},  $\delta(a)\left\langle a,v\right\rangle=0$. Hence $\left\langle a,v\right\rangle=0$, for every $v\in \mathfrak{n}_{\delta}$. We conclude by continuity.

Now if $\delta$ was not a root then $\mathfrak{s}_{\delta}=0$ and $\mathfrak{g}_{\delta}=\mathfrak{n}_{\delta}$. Thus $\mathfrak{a}$ and $\mathfrak{g}_{\delta}$ are orthogonal. Which implies that $\mathfrak{a}\subset \mathfrak{h}$. But $\mathfrak{h}$ is an ideal of $\mathfrak{p}$, so $\mathfrak{g}_{\delta}=\left[\mathfrak{g}_{\delta},\mathfrak{a}\right]\subset \mathfrak{h}$. This contradicts the fact that $\mathfrak{g}_{\delta}$ is paired with $\mathfrak{g}_{0}$. 

To finish we need to prove that $\mathfrak{s}_{\delta}\not\subset \mathfrak{h}$. If this was not the case then $\mathfrak{a}$ would be  orthogonal to $\mathfrak{g}_{\delta}$. Hence $\mathfrak{g}_{\delta}\subset \mathfrak{h}$ which contradicts again the fact that $\mathfrak{g}_{\delta}$ is paired with $\mathfrak{g}_{0}$.

\end{proof}

\subsection{The isotropy group is big}
From now and until the end we will suppose that  the non-compact semi-simple part $S_{nc}$ of $G$ is locally isomorphic to the Möbius group $\operatorname{SO}(1,n+1)$. In this case the Cartan Lie algebra $\mathfrak{a}$ is one dimensional and we have $\mathfrak{s}_{nc}=\mathfrak{s}_{-\alpha}\oplus\mathfrak{a}\oplus\mathfrak{m}\oplus\mathfrak{s}_{\alpha}$, where $\alpha$ is a positive root, $\mathfrak{a}\cong\mathbb{R}$, $\mathfrak{m}\cong\mathfrak{s}\mathfrak{o}(n)$, and $\mathfrak{s}_{-\alpha}\cong\mathfrak{s}_{\alpha}\cong \mathbb{R}^{n}$. Moreover, $\mathfrak{g}_{\pm\alpha}=\mathfrak{s}_{\pm\alpha}\oplus\mathfrak{n}_{\pm\alpha}$, $\mathfrak{g}_{0}=\mathfrak{a}\oplus\mathfrak{m}\oplus\mathfrak{s}_{c}\oplus \mathfrak{r}_{0}$, $\mathfrak{g}_{\beta}=\mathfrak{n}_{\beta}$ for every $\beta\neq 0,\pm\alpha$ and  $ \mathfrak{r}=\mathfrak{r}_{0}\oplus \bigoplus_{\beta\neq 0}\mathfrak{n}_{\beta}$.

In section~\ref{subsection} we saw that the isotropy group $H$ is contained in the algebraic group $P$ which turn out to be big i.e to contain the connected Lie groups $A$, $ S_{\alpha}$ and $N$. Our next result shows that the group $H$ itself is big:
\begin{prop}
\label{propositionsix}
The Lie algebra $\mathfrak{h}$ contains $\mathfrak{a}\oplus\mathfrak{s}_{\alpha}\oplus \bigoplus_{\beta\neq 0}\mathfrak{n}_{\beta}$.
\end{prop}

\begin{proof}

We have that $\mathfrak{a}\subset \mathfrak{h}$. Indeed, if $0$ is a degenerate weight then we are done. If not, then  $\delta$ is a root and  $\mathfrak{a}\subset \mathfrak{g}_{0}$ is orthogonal to  every $\mathfrak{g}_{\beta}$ with $\beta\neq \delta$. From the proof of Proposition \ref{proposition9} we know that $\mathfrak{a}$ and $\mathfrak{n}_{\delta}$ are orthogonal. Thus it remains to show that $\mathfrak{a}$ and $\mathfrak{s}_{\delta}$ are orthogonal. For that, let $x\in \mathfrak{s}_{\delta}$ then $\Theta(x)\in \mathfrak{s}_{-\delta}$ and $\left[x,\Theta(x)\right]\neq 0$ in $\mathfrak{a}$. Now  using Equation \ref{equationcinq} and the fact that one of $x$  or $\Theta(x)$ preserve $\left\langle .,.\right\rangle$, we get $\left\langle \operatorname{ad}_{x}(\Theta(x)), x \right\rangle=0$. But $\mathfrak{a}$ is one dimensional so it is orthogonal to $\mathfrak{s}_{\delta}$.\\
To end this proof, we have that $\mathfrak{h}$ is an ideal of $\mathfrak{p}$ and so $$\mathfrak{a}\oplus\mathfrak{s}_{\alpha}\oplus \bigoplus_{\beta\neq 0}\mathfrak{n}_{\beta}=\mathfrak{a}\oplus\left[\mathfrak{a}, \mathfrak{s}_{\alpha}\oplus \bigoplus_{\beta\neq 0}\mathfrak{n}_{\beta}\right]\subset \mathfrak{h}\oplus\left[\mathfrak{h},  \mathfrak{p}\right]\subset \mathfrak{h}.$$
\end{proof}
As a consequence we get:
\begin{cor}
\label{Corolary9}
If $0$ is a non-degenerate weight, then $\delta=-\alpha$.
\end{cor}

\subsection{A suitable modification of $\mathfrak{g}$}

We will show that $\mathfrak{g}$ admits a suitable modification $\mathfrak{g}'$. This allows us to  considerably simplify the proofs in the next section.  More precisely, we have:
\begin{prop}
\label{propcinq}
The solvable radical decomposes as a direct sum  $\mathfrak{r}=\mathfrak{r}_{1}\oplus \mathfrak{r}_{2}$,  where $\mathfrak{r}_{1}$ is a subalgebra commuting with the semi-simple factor $\mathfrak{s}$ and $\mathfrak{r}_{2}$ is an $\mathfrak{s}-$invariant linear subspace contained in $\mathfrak{h}$. In particular $\mathfrak{g}'=\mathfrak{s}\oplus\mathfrak{r}_{1}$ is a modification of $\mathfrak{g}$.
\end{prop}

To prove Proposition \ref{propcinq}, we need the following lemma:
\begin{lemm}
\label{lemmecinq}
We have $\left[\mathfrak{s},\mathfrak{n}\right]=\left[\mathfrak{s},\mathfrak{r}\right]\subset \mathfrak{h}$.
\end{lemm}

\noindent\textbf{Proof of Lemma \ref{lemmecinq}}. 
First we prove that $\left[\mathfrak{n}, \mathfrak{g}_{0}\right]\subset \mathfrak{h}$. For this, note that by the Jacobi identity and the fact that $\mathfrak{n}$ is an ideal of $\mathfrak{g}$, we have $\left[\bigoplus_{\beta\neq 0}\mathfrak{n}_{\beta}, \mathfrak{g}_{0}\right]=\bigoplus_{\beta\neq 0}\mathfrak{n}_{\beta}$ which in turn is a subset of $\mathfrak{h}$ by Proposition \ref{propositionsix}. Thus one  need to prove that $\left[\mathfrak{n}_{0}, \mathfrak{g}_{0}\right]\subset \mathfrak{h}$. We know that $\mathfrak{n}$ preserve the metric $\left\langle .,.\right\rangle$. So using Equation \ref{equationcinq} for $p\in \mathfrak{n}_{0}$, $u\in \mathfrak{g}_{0}$ and $v\in \mathfrak{g}_{\delta}$ gives us: $\left\langle\operatorname{ad}_{p}(u),v\right\rangle + \left\langle u, \operatorname{ad}_{p}(v)\right\rangle=0$. But once again by Jacobi identity, the fact that $\mathfrak{n}$ is an ideal of $\mathfrak{g}$ and  Proposition \ref{propositionsix} we have $\operatorname{ad}_{p}(v)\in \mathfrak{g}_{\delta}\cap \mathfrak{n}=\mathfrak{n}_{\delta}\subset \mathfrak{h}$. So $\left\langle\operatorname{ad}_{p}(u),v\right\rangle=0$, which means that  $\left[\mathfrak{n}_{0}, \mathfrak{g}_{0}\right]$ is orthogonal to $\mathfrak{g}_{\delta}$. Using the fact that $\left[\mathfrak{n}_{0}, \mathfrak{g}_{0}\right]\subset \mathfrak{g}_{0}$ and that $\mathfrak{g}_{0}$ is orthogonal to every $\mathfrak{g}_{\beta}$ for $\beta\neq \delta$ we get that $\left[\mathfrak{n}_{0}, \mathfrak{g}_{0}\right]\subset \mathfrak{h}$.\\

Next we have that $\mathfrak{s}_{c}\subset \mathfrak{g}_{0}$ thus $\left[\mathfrak{s}_{c},\mathfrak{n}\right]\subset \left[\mathfrak{g}_{0},\mathfrak{n}\right]\subset \mathfrak{h}$.\\

Finally we finish by proving that $\left[\mathfrak{s}_{nc},\mathfrak{n}\right]\subset \mathfrak{h}$. On the one hand we have, $$\left[\mathfrak{a}\oplus\mathfrak{s}_{\alpha}\oplus \mathfrak{m} ,\mathfrak{n}\right]\subset \left[\mathfrak{a}\oplus\mathfrak{s}_{\alpha}, \mathfrak{n}\right]+\left[\mathfrak{g}_{0},\mathfrak{n}\right]\subset \mathfrak{h}+\mathfrak{h}\subset \mathfrak{h}.$$
On the other hand, as $\mathfrak{s}_{nc}$ is semi-simple we have  by  Lemma \ref{Lemmaimportprime} that $\left[\mathfrak{s}_{nc},\mathfrak{n}\right]\subset \operatorname{Vect}\left(\left[\mathfrak{a}\oplus \mathfrak{m}\oplus \mathfrak{s}_{\alpha},\mathfrak{n}\right]\right)\subset \mathfrak{h}$.\\
\\
\noindent\textbf{Proof of Proposition \ref{propcinq}}.
The subalgebra $\left[\mathfrak{s},\mathfrak{n}\right]=\left[\mathfrak{s},\mathfrak{r}\right]$ is $\mathfrak{s}-$invariant, so it admits an $\mathfrak{s}-$invariant supplementary subspace $\mathfrak{r}_{1}'$ in $\mathfrak{r}$. But $\mathfrak{s}$ acts trivially on $\mathfrak{r}/\left[\mathfrak{s},\mathfrak{n}\right]$ and thus it acts trivially on $\mathfrak{r}_{1}'$. We take $\mathfrak{r}_{1}$ to be the $\mathfrak{s}-$invariant subalgebra generated by $\mathfrak{r}_{1}'$ (in fact the action of $\mathfrak{s}$ on $\mathfrak{r}_{1}$ is trivial).

It is clear that $\mathfrak{r}_{1}$ is a direct sum of $\mathfrak{r}_{1}'$ and  $\mathfrak{r}_{1}''$: an $\mathfrak{s}-$invariant subspace of  $\left[\mathfrak{s},\mathfrak{n}\right]$. Consider $\mathfrak{r}_{2}$ to be the supplementary of $\mathfrak{r}_{1}''$ in $\left[\mathfrak{s},\mathfrak{n}\right]=\left[\mathfrak{s},\mathfrak{r}\right]$. It is $\mathfrak{s}-$invariant and by Lemma \ref{lemmecinq} we have  $\mathfrak{r}_{2}\subset \mathfrak{h}$.

\section{The Möbius conformal group: a classification theorem}
This section is devoted to prove  Theorem \ref{theoremdeux}. We distinguish two situations: when $\mathfrak{m}$ is contained in $\mathfrak{h}$ and  when it is not. In this last one, we first consider the case where only the non-compact  semi-simple part $S_{nc}$ is non trivial. Then deduce from it the general case. From now and until the end we will assume, up to modification, that $\mathfrak{g}=\mathfrak{s}\oplus\mathfrak{r}_{1}$.

\subsection{The Frances-Melnick case}
We suppose that the sub-algebra $\mathfrak{m}$ is contained in $\mathfrak{h}$. Then we have the following proposition:
\begin{prop}
$M$ is conformally equivalent to the standard sphere $\mathbb{S}^{n}$ or  the Einstein universe $\operatorname{Ein}^{1,1}$.
\end{prop}
\begin{proof}
Assume first that  $\mathfrak{g}_{0}$ is contained in $\mathfrak{h}$. Then by Proposition \ref{proposition8}, $M$ is conformally flat and after modification, $\mathfrak{r}=0$. Moreover, $\mathfrak{g}/\mathfrak{h}\cong \mathfrak{s}_{-\alpha}$. This is because $\mathfrak{g}_{0}=\mathfrak{a}\oplus \mathfrak{m}\oplus\mathfrak{s}_{c}$ and $\left[\mathfrak{m}, X\right]=\mathfrak{s}_{-\alpha}$ for every $X\neq 0$ in $\mathfrak{s}_{-\alpha}$. Thus $M$ is conformally equivalent to  $\operatorname{SO}(1,n+1)/\operatorname{CO}(n)\ltimes \mathbb{R}^{n}\cong \mathbb{S}^{n}$.

Now suppose that $\mathfrak{g}_{0}$ is not in $\mathfrak{h}$. In this case $\mathfrak{g}_{-\alpha}=\mathfrak{g}_{\delta}$ is paired with $\mathfrak{g}_{0}$. But $\mathfrak{a}$, $\mathfrak{m}$ and $\mathfrak{n}_{\delta}$ are contained in $\mathfrak{h}$ so $\mathfrak{s}_{-\alpha}$ is paired with $\mathfrak{s}_{c}\oplus \left(\mathfrak{r}_{0}\cap \mathfrak{r}_{1}\right)$. Note that $\mathfrak{m}$ acts  on $\mathfrak{s}_{-\alpha}\oplus\left(\mathfrak{s}_{c}\oplus \left(\mathfrak{r}_{0}\cap \mathfrak{r}_{1}\right)\right)$ by preserving the pairing (in fact the action of $\mathfrak{m}$ preserves the metric $\left\langle .,.\right\rangle$). On the contrary for $n\geq 2$,  $\mathfrak{m}\cong \mathfrak{so}(n)$ acts trivially on $\mathfrak{r}_{0}\cap \mathfrak{r}_{1}$ and transitively on $\mathfrak{s}_{-\alpha}-\left\{0\right\}$, so $n=1$. As the metric is of type $(p,q)$, we conclude that the projection of $\mathfrak{s}_{c}\oplus \left(\mathfrak{r}_{0}\cap \mathfrak{r}_{1}\right)$ on $\mathfrak{g}/\mathfrak{h}$ is $\cong \mathbb{R}$. Thus, after modification $\mathfrak{g}=\mathfrak{so}(1,2)\oplus \mathbb{R}=\mathfrak{u}(1,1)$, $\mathfrak{h}=\mathfrak{a}\oplus \mathfrak{s}_{\alpha}=\mathbb{R}\oplus\mathbb{R}$ and hence $M$ is, up to cover, conformally equivalent to $\operatorname{Ein}^{1,1}$.

\end{proof}
\subsection{The non-compact semi-simple case}
Here we suppose that $\mathfrak{m}$ is not contained in $\mathfrak{h}$, the compact semi-simple part $\mathfrak{s}_{c}$  and the radical solvable part $\mathfrak{r}_{1}$ are both trivial. We will show:
\begin{prop}
\label{proposition89}
The pseudo-Riemannian manifold $M$ is conformally equivalent to $\operatorname{Ein}^{3,3}$
\end{prop}
By corollary \ref{Corolary9}, $\delta$ is a negative root. In particular $\delta=-\alpha$ and $\mathfrak{g}_{-\alpha}$ is paired with $\mathfrak{g}_{0} $. In addition $\mathfrak{g}=\mathfrak{s}_{-\alpha}\oplus\mathfrak{a}\oplus\mathfrak{m}\oplus\mathfrak{s}_{\alpha}$ and $\mathfrak{a}\oplus \mathfrak{s}_{\alpha} \subset\mathfrak{h}$. We have:
\begin{prop}
\label{proposition56}
The root space $\mathfrak{s}_{\delta}$ does not intersect $\mathfrak{h}$. In particular the metric is of type $(n,n)$.
\end{prop}
\begin{proof}
If it was the case then let $0\neq X\in \mathfrak{s}_{\delta}\cap\mathfrak{h}$. We have $\left[\left[X,\mathfrak{s}_{-\delta}\right],X\right]=\mathfrak{s}_{\delta}$ so $\mathfrak{s}_{\delta}\subset \mathfrak{h}$. This contradicts the fact that $\mathfrak{g}_{\delta}$ is paired with $\mathfrak{g}_{0}$.
\end{proof}

Consider the bracket $\left[.,.\right]:\mathfrak{s}_{\alpha}\times \mathfrak{s}_{-\alpha}\longrightarrow \mathfrak{a}\oplus \mathfrak{m}$. Denote by $.\wedge . :\mathfrak{s}_{\alpha}\times \mathfrak{s}_{-\alpha}\longrightarrow \mathfrak{m}$ and $.\vee . :\mathfrak{s}_{\alpha}\times \mathfrak{s}_{-\alpha}\longrightarrow \mathfrak{a}$ its projections on $\mathfrak{m}$ and $\mathfrak{a}$ respectively. Direct computations give us:
\begin{lemm}\noindent
\label{Lemme57}
\begin{enumerate}
\item $\forall X\in \mathfrak{s}_{\alpha}$, $\forall x\in \mathfrak{s}_{-\alpha}$: $X\vee x=-\Theta(x)\vee \Theta(X)$.
\item $\forall X\in \mathfrak{s}_{\alpha}$, $\forall x\in \mathfrak{s}_{-\alpha}$, $\forall y\in \mathfrak{s}_{-\alpha}$: $\left[X\wedge x, y\right]=\left[X\wedge y, x\right]-\left[\Theta(x)\wedge y, \Theta(X)\right]$
\end{enumerate}
\end{lemm}

The Cartan involution  identifies $\mathfrak{s}_{\alpha}$ and $\mathfrak{s}_{-\alpha}$, which when identified with $\mathbb{R}^{n}$, $\mathfrak{m}$ acts on them as $\mathfrak{so}(n)$. In this case, the map $.\vee .$ can be seen as a bilinear symmetric map from $\mathbb{R}^{n}\times \mathbb{R}^{n}$ to $\mathbb{R}^{n}$, and when composed with $\alpha$ gives rise to an $\mathfrak{m}-$invariant scalar product $\left\langle .,. \right\rangle_{0}$ on $\mathbb{R}^{n}$. Moreover, by Lemma \ref{Lemme57}, for every $x,X\in \mathbb{R}^{n}$, $X\wedge x$ is the antisymmetric endomorphism of $\mathbb{R}^{n}$ defined by $X\wedge x(y)=\left\langle X,  y\right\rangle_{0}x-
\left\langle x,  y\right\rangle_{0} X $. 

Let $x,X\in \mathbb{R}^{n}$ and consider $P$ the plane generated by $x,X$. Then $X\wedge x$ when seen as element of $\mathfrak{m}\cong\mathfrak{so}(n)$ is the infinitesimal generator 
of a one parameter group acting trivially on the orthogonal $P^{\perp}$ of $P$ with respect to the scalar product 
$\left \langle .,  .\right \rangle_{0}$. Hence $X\wedge x\in \mathfrak{so}(P)$. More generally:

\begin{prop}
\label{Proposition57}
Let $E$ be a linear subspace of $\mathbb{R}^{n}$ and let $x\in E$. Consider $\mathfrak{c}$ the Lie subalgebra of $\mathfrak{so}(n)$ generated by $\left\{X\wedge x / X\in E \right\}$. Then $\mathfrak{c}$ equals the Lie algebra linearly generated by $\left\{ X \wedge X^\prime / X, X^\prime \in E\right\}$, which in turn equals 
$\mathfrak{so}(E)$,  the Lie algebra of orthogonal transformations preserving $E$ and acting trivially on its orthogonal (with respect to $\left \langle .,  .\right \rangle_{0}$).
\end{prop}
\begin{proof}
First we have $\mathfrak{c}(E)\subset E$ and hence $\mathfrak{c}\subset \mathfrak{so}(E)$. It is then sufficient to prove that $\mathfrak{c} $ and $\mathfrak{so}(E) $ have the same dimensions. For that let $\left\{x, X_{2},...,X_{k}\right\}$ be a basis of $E$. Note that $\left\{X_{2}\wedge x,...,X_{k}\wedge x, \left[X_{i}\wedge x,X_{j}\wedge x\right],\mbox{~for~} 2\leq i<j \leq k\right\}$ are linearly independent. Thus $\mathfrak{c}=\mathfrak{so}(E)$.
\end{proof}

For every $x\neq 0\in \mathfrak{s}_{-\alpha}$  consider: $$Z_{x}=\left\{ X\neq 0 \in \mathfrak{s}_{\alpha}, \mbox{~such that~} \left[X,x\right]\in \mathfrak{h}\right\}.$$
Denote $\overline{\Theta(Z_{x})}$ the projection of $\Theta(Z_{x})$ in $\mathfrak{g}/\mathfrak{h}$. Then:
\begin{prop}
\label{Proposition99}
The family $\left\{\overline{\Theta(Z_{x})}\backslash\lbrace{0}\rbrace\right\}_{x\in \mathfrak{s}_{-\alpha}}$ form a partition of $\mathfrak{g}/\mathfrak{h}$. More precisely: $$\overline{\Theta(Z_{x})}\cap\overline{\Theta(Z_{y})}=\left\{0\right\}\Longleftrightarrow x\notin \Theta(Z_{y})\Longleftrightarrow y\notin \Theta(Z_{x}).$$

\end{prop}
\begin{proof}
By Proposition \ref{Proposition57} we have, $$\left[Z_{x},\Theta(Z_{x})\right]=\mathfrak{a}\oplus\operatorname{alg}\left(\left\{X\wedge x / X\in  Z_{x}\right\}\right)\subset \mathfrak{h}.$$ 
This implies that  $$\Theta(Z_{x})=\Theta(Z_{y})\Longleftrightarrow x\in \Theta(Z_{y})\Longleftrightarrow y\in \Theta(Z_{x}).$$ 
Hence, the projections $\left\{\overline{\Theta(Z_{x})}\backslash\lbrace{0}\rbrace\right\}_{x\in \mathfrak{s}_{-\alpha}}$ form a partition of $\mathfrak{g}/\mathfrak{h}$.
\end{proof}

Next we prove:
\begin{prop}

The pseudo-Riemannian manifold $M$ is conformally flat. 
\end{prop}
\begin{proof}
We need to prove that the Weyl tensor $\operatorname{W}$ (or the Cotton tensor $\operatorname{C}$ if the dimension of $M$ is $3$) vanishes. Actually we will just make use of their conformal invariance property. Namely: if $f$ is a conformal transformation of $M$ then,

\begin{equation}
\label{equationdix}
d_{x}f\operatorname{W}(X,Y,Z)=\operatorname{W}(d_{x}f (X),d_{x}f(Y),d_{x}f(Z))
\end{equation}

We denote by $\bar{x}$ the projection in $\mathfrak{g}/\mathfrak{h}$ of  an element $x\in \mathfrak{g}$. A direct application of Equation \ref{equationdix} gives us:
\begin{enumerate}
\item $\operatorname{W}(\bar{x},\bar{y},\bar{z})=0$ for every $x,y,z\in \mathfrak{s}_{-\alpha}$;
\item $\operatorname{W}(\bar{x},\bar{y},\overline{m})=0$ for every $x,y\in \mathfrak{s}_{-\alpha}$ and every $m\in \mathfrak{m}$;
\item $\left[X, \operatorname{W}(\bar{x},\overline{m}_{1},\overline{m}_{2})\right]= \operatorname{W}(\left[X,\bar{x}\right],\overline{m}_{1},\overline{m}_{2})$ for every $X\in \mathfrak{s}_{\alpha}$, $x\in \mathfrak{s}_{-\alpha}$ and every $m_{1}, m_{2}\in \mathfrak{m}$.
\end{enumerate}

\noindent Let $x\in \mathfrak{s}_{-\alpha}$, $m_{1}, m_{2}\in \mathfrak{m}$. Then, from Equation \ref{equationdix} we obtain: $$\left[\Theta(x), \operatorname{W}(\bar{x},\overline{m}_{1},\overline{m}_{2})\right]= \operatorname{W}(\left[\Theta(x),\bar{x}\right], \overline{m}_{1}, \overline{m}_{2})=0.$$ In other words  $$\operatorname{W}(\bar{x},\overline{m}_{1},\overline{m}_{2})\in \overline{\Theta(Z_{x})}.$$ 

\noindent Now let $x,y\in \mathfrak{s}_{-\alpha}$, $X\in \mathfrak{s}_{\alpha}$ and $m\in\mathfrak{m}$. Then again Equation \ref{equationdix} gives us: $$\operatorname{W}(\bar{x},\left[X,\bar{y}\right],\overline{m})+\operatorname{W}(\left[X,\bar{x}\right],\bar{y},\overline{m})=0.$$ 

\noindent But $\operatorname{W}(\bar{x},\left[X,\bar{y}\right],\overline{m})\in \overline{\Theta(Z_{x})}$ and $\operatorname{W}(\left[X,\bar{x}\right],\bar{y},\overline{m})\in \overline{\Theta(Z_{y})}$. Thus, Proposition \ref{Proposition99} gives us:
\begin{enumerate}
\item If $y\notin \Theta(Z_{x})$ then $\operatorname{W}(\bar{x},\left[X,\bar{y}\right],\overline{m})=0$;
\item In the case $y\in \Theta(Z_{x})$ and $X\in Z_{x}$, we have $\operatorname{W}(\bar{x},\left[X,\bar{y}\right],\overline{m})=0$ 
\item If $y\in \Theta(Z_{x})$ and $X\notin Z_{x}$. Then because $\Theta(X)\notin \Theta(Z_{x})$ we have: $$\operatorname{W}(\bar{x},\left[X,\bar{y}\right],\overline{m})=\operatorname{W}(\bar{x},\left[\Theta(y),\overline{\Theta(X)}\right],\overline{m})=0.$$
\end{enumerate}
So  as a conclusion we get $\operatorname{W}=0$.
\end{proof}

We finish this section by proving Proposition \ref{proposition89}:

\noindent\textbf{Proof of Proposition \ref{proposition89}}. First note that if $n = 1$  then $\mathfrak{m}= 0$. Thus we assume $n\geq 2$. So far we have seen that $M = \operatorname{SO}(1,n+1) / H$ is 
a conformally flat pseudo-Riemannian manifold of signature $(n,n)$. Since the Lie algebra $\mathfrak{h}$ contains $\mathfrak{a} +\mathfrak{s}_{\alpha} $, the group $H^{°}$ is cocompact in $\operatorname{SO}(1, n+1)$. 
Therefore $\operatorname{SO}^{°}(1, n+1) / H^{°}$ is connected and  compact, with a connected isotropy and hence simply connected. As $M$ is connected, it covers $\operatorname{SO}^{°}(1, n+1) / H^{°}$ and thus equals it. 

On the one hand, the Einstein universe $\operatorname{Ein}^{n, n}$ is simply connected. Thus  $M$ is identified to $\operatorname{Ein}^{n, n}$.  So $\operatorname{SO}(1,n+1) $ acts transitively on $\operatorname{Ein}^{n,n}$ with isotropy $H$. By Montgomery Theorem \cite[Theorem~A]{Montgomery} any maximal compact subgroup in $\operatorname{SO}(1,n+1)$, e.g. $K_{2}=\operatorname{SO}(n+1)$,  acts transitively on $\mathbb{S}^{n} \times \mathbb{S}^{n}$ the two fold cover of $\operatorname{Ein}^{n,n}$.

On the other hand, the  conformal group of $\operatorname{Ein}^{n,n}$ is $\operatorname{SO}(n+1, n+1)$.  A maximal compact subgroup of it is $K_{1}= \operatorname{SO}(n+1) \times \operatorname{SO}(n+1)$. Up to conjugacy , we can assume $K_{2} \subset K_{1}$. Therefore,  $ K_{2} = \operatorname{SO}(n+1)$ acts 
via a homomorphism $\rho = (\rho_{1}, \rho_{2}):  \operatorname{SO}(n+1) \to \operatorname{SO}(n+1) \times \operatorname{SO}(n+1)$. \\
If $\operatorname{SO}(n+1)$ is simple, then: \\
-  either $\rho_{1}$ or $\rho_{2}$ is trivial and the other one is bijective, in which case $\rho (\operatorname{SO}(n+1))$ does not act transitively on $\mathbb{S}^{n} \times \mathbb{S}^{n}$, \\
- or both are bijective, and $\rho (\operatorname{SO}(n+1)$ is up to conjugacy 
in $\operatorname{SO}(n+1) \times \operatorname{SO}(n+1)$ the diagonal $\left\{(g, g)/ g \in \operatorname{SO}(n+1)\right\}$. The latter, too,  does not act transitively on $\mathbb{S}^{n} \times \mathbb{S}^{n}$.

Hence $\operatorname{SO}(n+1)$ must be non-simple which implies $n = 1$ or $n = 3$. but $n = 1$ was excluded, and then remains exactly the case $n = 3$,  for which $M$ is conformally equivalent to $\operatorname{Ein}^{3,3}$.

\subsection{The general case}
In this section we will show Theorem \ref{theoremdeux} in the general case. We suppose that $\mathfrak{g}=\mathfrak{s}_{-\alpha}\oplus\mathfrak{a}\oplus\mathfrak{m}\oplus\mathfrak{s}_{\alpha}\oplus\mathfrak{s}_{c}\oplus \mathfrak{r}_{1}$. Let us denote by $\mathfrak{m}_{0}=\mathfrak{m}\cap \mathfrak{h}$ so that $\mathfrak{so}(1,n+1)\cap \mathfrak{h}= \mathfrak{a}\oplus \mathfrak{s}_{\alpha}\oplus \mathfrak{m}_{0}$. A priori the subalgebra $\mathfrak{m}_{0}$ could be of any dimension in $\mathfrak{m}$. Nevertheless the hypothesis $\mathfrak{m}\not \subset \mathfrak{h}$ restricts drastically the possibilities. So we have:
\begin{prop}
\label{propositionderniere}
The subalgebra $\mathfrak{m}_{0}$ has codimension $n$ in $\mathfrak{m}$.
\end{prop}
\begin{proof}
If $n=2$ then $\mathfrak{m}=\mathfrak{so}(2)$. Hence $\left[p,\mathfrak{s}_{\alpha}\right]=\mathfrak{a}\oplus \mathfrak{m}$ for any non null $p\in \mathfrak{s}_{-\alpha} $. Recall that $\mathfrak{s}_{-\alpha} $ preserves the metric so by applying Equation \ref{equationcinq} for $p=v\in \mathfrak{s}_{-\alpha}$, $u\in \mathfrak{s}_{\alpha}$ we get $\left\langle \mathfrak{s}_{-\alpha},\mathfrak{m}\right\rangle=0$. Thus  $\mathfrak{m}\subset \mathfrak{h}$ which contradicts our hypothesis.

Assume that $n\geq 3$ and suppose that  $\mathfrak{m}_{0}$ has codimension less then $n-1$. Denote by $M_{0}$ the connected subgroup of $\operatorname{SO}(n)$ corresponding to $\mathfrak{m}_{0}$. 

If the action of $M_{0}$ on $\mathfrak{s}_{-\alpha}\cong \mathbb{R}^{n}$ is reducible then  $M_{0}$ preserves the splitting $\mathbb{R}^{d}\times \mathbb{R}^{n-d}$ and hence is contained in $\operatorname{SO}(d)\times \operatorname{SO}(n-d)$. Thus $M_{0}$ has codimension bigger than the codimension of $\operatorname{SO}(d)\times \operatorname{SO}(n-d)$ which in turn achieves its minimum if $d=1$ or $n-d=1$ and hence $M_{0}=\operatorname{SO}(n-1)$. One can identify $\mathfrak{m}_{0}$ with $\mathfrak{so}(E)$ for some $n-1$ dimensional linear subspace $E$ of $\mathfrak{s}_{-\alpha}$. Let then $e\in \mathfrak{s}_{-\alpha}$ such that $\mathfrak{s}_{-\alpha}=\mathbb{R}e\oplus E$. Fix a non zero element $x\in \Theta(E)$, we have $\left\langle \operatorname{ad}_{x}(e), X\right\rangle+\left\langle e, \operatorname{ad}_{x}X\right\rangle=0$ for every $X\in \mathfrak{s}_{-\alpha}$ and so in particular $\left\langle e, \operatorname{ad}_{x}e\right\rangle=0$. In addition by Proposition \ref{Proposition57}, $\left[E,\Theta(E)\right]= \mathfrak{a}\oplus \mathfrak{m}_{0}\subset \mathfrak{h}$ thus $\left\langle \operatorname{ad}_{x}e,X\right\rangle=0$ for every $X\in E$ and hence  $\operatorname{ad}_{x}e$ is orthogonal to $\mathfrak{s}_{-\alpha}$. This implies that $x\wedge e\in \mathfrak{h}\cap \mathfrak{m}=\mathfrak{m}_{0}=\mathfrak{so}(E)$ which contradicts the fact that $x\wedge e$ is the infinitesimal rotation of the plane $\mathbb{R}e\oplus \mathbb{R}x$.

The last case to consider is when $M_{0}$ acts irreducibly. Let $m\in \mathfrak{m}_{0}$, $X\in \mathfrak{s}_{-\alpha}$ and $y\in  \mathfrak{s}_{c}\oplus \mathfrak{r}_{1}$  then $\left\langle \operatorname{ad}_{m}(X), y\right\rangle+\left\langle X, \operatorname{ad}_{m}y\right\rangle=0$. But $\operatorname{ad}_{m}y=0$ and hence  $\mathfrak{s}_{c}\oplus \mathfrak{r}_{1}$ is orthogonal to $\left[\mathfrak{m}_{0},\mathfrak{s}_{-\alpha}\right]$ which is equal to $\mathfrak{s}_{-\alpha}$ by irreducibility. Thus $\mathfrak{s}_{c}\oplus \mathfrak{r}_{1}\subset \mathfrak{h}$ and we are in the non-compact semi-simple case. Therefore $n=3$ and $\mathfrak{m}\cong \mathfrak{so}(3)$. Non trivial Sub-algebras of  $\mathfrak{so}(3)$ have dimension one and are reducible. So the only left possibility is $\mathfrak{m}_{0}=\mathfrak{m}\cong \mathfrak{so}(3)$ which show that $\mathfrak{m}\subset \mathfrak{h}$ and this is a contradiction.
\end{proof}
\noindent\textbf{End of Proof of Theorem  \ref{theoremdeux}}.
By Proposition \ref{propositionderniere}, $\mathfrak{m}_{0}$ is of codimension $n$ in $\mathfrak{m}$. But $\mathfrak{s}_{-\alpha}$ is paired with $\mathfrak{g}_{0}=\mathfrak{a}\oplus\mathfrak{m}\oplus\mathfrak{s}_{c}\oplus \left(\mathfrak{r}_{1}\cap \mathfrak{r}_{0}\right)$. Thus $\mathfrak{s}_{c}\oplus \left(\mathfrak{r}_{1}\cap \mathfrak{r}_{0}\right)\subset \mathfrak{h}$ and we are also in the non-compact semi-simple case. Therefore $n=3$ and $M$ is conformally equivalent to $\operatorname{Ein}^{3,3}$.

\cite{*}

\bibliographystyle{Plain}

\bibliography{Pseudo_Conformal_Actions_Mobius_Group}

\end{document}